\documentclass{amsart}
\usepackage{amsfonts,amsmath,amsthm,amssymb,mathrsfs,bbm,mathtools,enumitem}
\usepackage{lmodern}
\usepackage[T1]{fontenc}
\usepackage[pdfpagelabels,hyperfootnotes=false,plainpages=false]{hyperref}
\usepackage{color,soul,verbatim}

\theoremstyle{plain}
\newtheorem{theorem}{Theorem}
\newtheorem{claim}[theorem]{Claim}

\newtheorem{corollary}[theorem]{Corollary}
\newtheorem{lemma}[theorem]{Lemma}

\newtheorem{question}[theorem]{Question}

\theoremstyle{definition}

\newtheorem{definition}[theorem]{Definition}

\theoremstyle{plain}
\newcommand{\thistheoremname}{}
\newtheorem{genericthm}[theorem]{\thistheoremname}

\newtheorem*{genericthm*}{\thistheoremname}
\newenvironment{namedthm*}[1]
  {\renewcommand{\thistheoremname}{#1}%
   \begin{genericthm*}}
  {\end{genericthm*}}

\newcommand{\de}{\delta}

\newcommand{\al}{\alpha}
\newcommand{\ga}{\gamma}

\newcommand{\ka}{\kappa}
\newcommand{\eps}{\epsilon}
\newcommand{\R}{\mathbb{R}}
\newcommand{\Q}{\mathbb{Q}}
\newcommand{\Qplus}{\mathbb{Q}_+}
\newcommand{\Z}{\mathbb{Z}}
\newcommand{\N}{\mathbb{N}}

\newcommand\restr[2]{{ \left.\kern-\nulldelimiterspace #1 \right|_{#2}}}

\newcommand{\intpart}[1]{\left\lfloor #1 \right\rfloor}
\newcommand{\fracpart}[1]{\left\{ #1 \right\}}
\newcommand{\fracdist}[1]{\left\| #1 \right\|}
\newcommand{\bothceil}[1]{\left\lceil #1 \right\rceil}
\newcommand{\bothfloor}[1]{\left\lfloor #1 \right\rfloor}

\newcommand{\psa}{\text{PS}(\al)}
\newcommand{\ps}[1]{\text{PS}(#1)}

\newcommand{\varchange}[2]{t_{#1} \left({#2} \right)}
\newcommand{\varchangeempty}[1]{t_{#1}}

\newcommand{\intone}{\theta_1}
\newcommand{\inttwo}{\theta_2}
\newcommand{\intthree}{\theta_3}

\newcommand{\bigQ}{\mathcal{Q}}
\newcommand{\bigS}{\mathcal{S}}
\newcommand{\bigT}{\mathcal{T}}

\newcommand{\cent}{\varphi}
\newcommand{\cize}{\psi}
\newcommand{\openint}{J}

\newcommand{\supnormopenint}[1]{\|{#1}\|_{\overline{\openint},\infty}}

\newcommand{\finitesums}{\text{FS}}

\newcommand{\ip}{\text{IP}}
\newcommand{\vip}{\text{VIP}}

\newcommand{\ipzero}{\text{IP}_0}

\title[A Khintchine-type theorem and linear equations in P-S sequences]{A perturbed Khintchine-type theorem and solutions to linear equations in Piatetski-Shapiro sequences}
\author[D. Glasscock]{Daniel Glasscock}
\address{Department of Mathematics, Northeastern University, 360 Huntington Ave., Boston, MA 02115}
\email{dgglasscock@neu.edu}

\begin{document}

\begin{abstract} Our main result concerns a perturbation of a classic theorem of Khintchine in Diophantine approximation. We give sufficient conditions on a sequence of positive real numbers $(\cize_n)_{n \in \N}$ and differentiable functions $(\varphi_n: J \to \R)_{n \in \N}$ so that for Lebesgue-a.e. $\theta \in J$, the inequality $\| n\theta + \varphi_n(\theta) \| \leq \cize_n$ has infinitely many solutions. The main novelty is that the magnitude of the perturbation $|\varphi_n(\theta)|$ is allowed to exceed $\cize_n$, changing the usual ``shrinking targets'' problem into a ``shifting targets'' problem. As an application of the main result, we prove that if the linear equation $y=ax+b$, $a, b \in \R$, has infinitely many solutions in $\N$, then for Lebesgue-a.e. $\al > 1$, it has infinitely many or finitely many solutions of the form $\lfloor n^\al \rfloor$ according as $\al < 2$ or $\al > 2$. \end{abstract}



\maketitle

\section{Introduction}

Denote by $\{x\}$ the fractional part of $x \in \R$ and by $\|x\|$ the distance from $x$ to the integers. The main result in this paper concerns solutions in the positive integers $n \in \N$ to the system
\begin{align}\label{eqn:systemnewintro}\begin{dcases} \big\|\theta n + \cent_n(\theta) \big\| \leq \cize_n \\ \fracpart {\rho_n(\theta)} \in I \end{dcases},\end{align}
where $(\cize_n)_{n \in \N}$ is a sequence of positive real numbers, $I$ is an interval in $[0,1)$, $(\cent_n)_{n \in \N}$ and $(\rho_n)_{n \in \N}$ are sequences of differentiable functions on an interval $J \subseteq \R$, and $\theta \in J$.  More precisely, we show that under the right assumptions on these quantities, the system in (\ref{eqn:systemnewintro}) has infinitely many solutions for Lebesgue-almost every $\theta \in J$.

Because some of the assumptions are quite technical, we will state only the following special case of the main theorem here in the introduction. The theorem is stated in its full generality in Section \ref{sec:proofofmainthm}.

\begin{theorem}[Special case]\label{thm:mainthm}
Let $0 < \sigma < 1$, and put $\cize_n = 1/n^\sigma$. Let $J \subseteq \R$ be a non-empty, open interval. Let $(\cent_n)_{n\in\N}$ and $(\rho_n)_{n \in \N}$ be sequences of $C^1$ functions on $\overline{J}$ satisfying:
\begin{enumerate}
\item $\displaystyle \sup_{n \in \N} \supnormopenint {\cent_n} < \infty$;
\item $\displaystyle \big(\supnormopenint {\cent_n'}\big)_{n\in\N}$  is eventually non-increasing;
\item there exists $\eps > 0$ such that $\lim_{n \to \infty} \supnormopenint {\cent_n'} n^{2\sigma -1 + \eps } = 0$; and
\item $\lim_{n \to \infty} \supnormopenint {\rho_n'} n^{-(1+\sigma)} = 0$.
\end{enumerate}
In addition, suppose that $\rho_n$ satisfies an equidistribution condition.\footnote{Condition (C1) from the statement of Theorem \ref{thm:mainthm} in its full generality in Section \ref{sec:proofofmainthm}.} Let $I \subseteq [0,1)$ be a non-empty interval. For Lebesgue-a.e. $\theta \in \openint$, the system in (\ref{eqn:systemnewintro}) has infinitely many solutions.
\end{theorem}

Theorem \ref{thm:mainthm} is best contextualized as a ``perturbed'' and ``twisted'' variant of Khintchine's classic result in Diophantine approximation \cite{khintchinediophantineapprox}; see also \cite[Theorem 2.2]{harmanbook}.  Under the assumptions that $(n \psi_n)_{n \in \N}$ is non-increasing and that $\sum_{n=1}^\infty \psi_n = \infty$, Khintchine proved that the inequality $\|\theta n\| \leq \cize_n$ has infinitely many solutions for a.e. $\theta \in \R$. Interpreted in the language of dynamics, there are infinitely many times $n \in \N$ for which the point $0$ under the rotation $x \mapsto x+ \theta$ on the torus $\R/\Z$ lands into the ``shrinking target'' given by $(-\cize_n,\cize_n)$.

The quantity $\varphi_n(\theta)$ signifies a perturbation of this rotation.  In the case that the magnitude of the perturbation $|\varphi_n(\theta)|$ is less than the accuracy of the approximation $\cize_n$, a simple application of the triangle inequality eliminates the perturbation and reduces the inequality in (\ref{eqn:systemnewintro}) to one covered by Khintchine's theorem. This is one of the primary strengths of our main result: the magnitude of the perturbation is allowed to exceed $\cize_n$. Dynamically, the targets we consider take the form $(-\varphi_n(\theta)-\cize_n,-\varphi_n(\theta)+\cize_n)$; they are shrinking in length but no longer nested, leading to a ``shifting targets'' problem.

Of course, solving such a shifting targets problem would be impossible without strong restrictions on the nature of the perturbation. Indeed, with no restrictions, one could define $\varphi_n(\theta) = -\theta n + \xi_n(\theta)$ for an arbitrary function $\xi_n: \overline{J} \to \R$ and prove that there are infinitely many solutions to the inequality $\|\xi_n(\theta)\| \leq \cize_n$.   There are few theorems to the author's knowledge that supply sufficient conditions on a sequence of functions $(\xi_n)_{n \in \N}$ to recover such a general result; for one such result, see \cite{cassels}.  In the case that $\cize_n = 1/n^\sigma$, conditions (1), (2), and (3) are sufficient restrictions on the perturbation to solve the shifting targets problem.

The twist to this shifting targets problem is provided by the second condition in (\ref{eqn:systemnewintro}). The inclusion of this condition is primarily motivated by our main application and does not cause a great deal of added difficulty in the proof of Theorem \ref{thm:mainthm}. Setting $I = [0,1)$ in the statement of the theorem allows one to entirely disregard this twist, and it is the author's belief that the resulting theorem remains novel and potentially useful outside the scope of this work.


A illustrative example to which Theorem \ref{thm:mainthm} applies is gotten by putting $\cize_n = 1/n^\sigma$ and $\varphi_n(\theta) = \sin(n^\kappa \theta) / n^\delta$ where $0 < \delta < \min(\sigma, \kappa)$ and $\kappa + 2\sigma -\delta < 1$.  (Take, for example, $\sigma = \kappa = 1/3$ and $\delta = 1/6$.) In this case, the perturbation $\varphi_n$ oscillates, and its magnitude is greater than $\cize_n$. The proof of Theorem \ref{thm:mainthm} makes use of the monotonicity of $\supnormopenint {\cent_n'}$ and its decay to find solutions to (\ref{eqn:systemnewintro}).

At the heart of the proof is Lemma \ref{lem:hardersolutionstoinequality}, a result on the number of solutions to a perturbed version of the inequality $|qr-ps| \leq L$, one that arises frequently in the theory.  Key in counting solutions to this Diophantine inequality is the fact that the quantities involved are integers, making it amenable to techniques in basic number theory.  The perturbed inequality no longer has this feature; counting its solutions is the main technical difficulty overcome in this work.

The primary motivation for Theorem \ref{thm:mainthm} came from the desire to improve on the main result in \cite{glasscocksolutionstops} concerning solutions to linear equations in Piatetski-Shapiro sequences. A \emph{Piatetski-Shapiro sequence} is a sequence of the form $(\intpart {n^\al})_{n \in \N}$ for non-integral $\al > 1$ where $\bothfloor x$ denotes the integer part (or floor) of $x \in \R$. The image of $n \mapsto \intpart {n^\al}$ is denoted by $\psa$. We will say that the linear equation
\begin{align}\label{eqn:main}y=ax+b, \quad a, b \in \R\end{align}
is \emph{solvable} in $\psa$ if there are infinitely many distinct pairs $(x,y) \in \psa \times \psa$ satisfying (\ref{eqn:main}), and \emph{unsolvable} otherwise. This terminology extends as expected to solving equations and systems of equations in other subsets of $\N$.

\begin{theorem}\label{thm:mainapplication}
Suppose $a, b \in \R$, $a \not\in \{0,1\}$, are such that (\ref{eqn:main}) is solvable in $\N$. For Lebesgue-a.e. $\al > 1$, the equation (\ref{eqn:main}) is solvable or unsolvable in $\psa$ according as $\al < 2$ or $\al > 2$.
\end{theorem}

This theorem strengthens the main result in \cite{glasscocksolutionstops} by removing a restriction on $a$ and $b$; it is Theorem \ref{thm:mainthm} that allows us to overcome that restriction. Theorem \ref{thm:mainapplication} brings the the main result in \cite{glasscocksolutionstops} to its proper conclusion; the reader is encouraged to consult that paper for the motivation on finding solutions to linear equations and more general combinatorial structure in Piatetski-Shapiro sequences. The following corollary demonstrates how Theorem \ref{thm:mainapplication} allows us to further that goal. Denote by $\Qplus$ the set of positive rational numbers.

\begin{corollary}\label{cor:maincortwo}
For Lebesgue-a.e. $\al > 1$, the limiting quotient set of $\psa$,
\begin{align}\label{eqn:limitquotentset}\bigcap_{N=1}^\infty \left\{ \ \frac mn \ \middle| \ m,n \in \psa \cap [N,\infty) \right\},\end{align}
is equal to $\Qplus$ or $\{1\}$ according as $\al < 2$ or $\al > 2$.
\end{corollary}

In \cite[Section 5]{glasscocksolutionstops}, a number of questions were posed about further combinatorial structure in Piatetski-Shapiro sequences. Our Theorem \ref{thm:mainapplication} resolves the first of those questions, but to the author's knowledge, the other questions remain open.  The interested reader is encouraged to consult that paper for further ideas and references.

This paper is organized as follows. We begin by establishing notation and preliminary lemmas in Section \ref{sec:auxlemmas}, then we prove Theorem \ref{thm:mainthm} in Section \ref{sec:proofofmainthm}. The main application to Piatetski-Shapiro sequences, Theorem \ref{thm:mainapplication}, and Corollary \ref{cor:maincortwo} are proven in Section \ref{sec:reducttodiophantineapprox}.

\section{Notation and auxillary lemmas}\label{sec:auxlemmas}

For $x \in \R$, denote the distance to the nearest integer by $\fracdist x$, the fractional part by $\fracpart x$, the integer part (or floor) by $\bothfloor x$, and the ceiling by $\bothceil x := -\bothfloor {-x}$. Denote the Lebesgue measure on $\R$ by $\lambda$, and denote the set of those points belonging to infinitely many of the sets in the sequence $(E_n)_{n \in \N}$ by $\limsup_{n \to\infty} E_n$. Given two positive-valued functions $f$ and $g$, we write $f \ll_{a_1,\ldots,a_k} g$ or $g \gg_{a_1,\ldots,a_k} f$ if there exists a constant $K > 0$ depending only on the quantities $a_1, \ldots, a_k$ for which $f(x) \leq K g(x)$ for all $x$ in the domain common to both $f$ and $g$. The supremum norm of a real-valued continuous function $\varphi$ on an interval $J$ is denoted by $\|\varphi\|_{J,\infty}$.

The following five lemmas play a key role in the proof of Theorem \ref{thm:mainthm}. The first three are standard.  The fourth and fifth concern solutions to the fundamental inequality $|qr-ps| \leq L$ and a perturbation of it. On a first reading, it would be possible to skip to the statement and proof of the main theorem in the next section and use the remainder of this section simply as a reference.

\begin{lemma}[{\cite[Lemma 1.6]{harmanbook}}]\label{lem:densitylemma}
Let $I \subseteq \R$ be an interval and $A \subseteq I$ be measurable.  If there exists a $\de > 0$ such that for every sub-interval $I' \subseteq I$, $\lambda (A \cap I') \geq \de \lambda (I')$, then $A$ is of full measure in $I$: $\lambda( I \setminus A ) = 0$.
\end{lemma}

\begin{lemma}[{\cite[Lemma 2.3]{harmanbook}}]\label{lem:advancedBC}
Let $(X,\mathcal{B},\mu)$ be a measure space with $\mu(X)<\infty$. If $(G_n)_{n \in \N} \subseteq \mathcal{B}$ is a sequence of subsets of $X$ for which $\sum_{n=1}^\infty \mu(G_n) = \infty$, then
\[\mu\left(\limsup_{n\to\infty} G_n \right) \geq \limsup_{N \to \infty} \left( \sum_{n=1}^N \mu(G_n) \right)^2 \left( \sum_{n,m=1}^N \mu(G_n \cap G_m) \right)^{-1}.\]
\end{lemma}

In the following lemmas, an interval $\openint \subseteq \R$ \emph{reduced modulo 1} means the set  $\big\{ \{x\} \ \big| \ x \in J \big\} \subseteq [0,1)$.

\begin{lemma}\label{lem:discrepofnalpha}
Let $\openint \subseteq \R$ be a non-empty, open interval, and let $\openint' \subseteq [0,1)$ be $\openint$ reduced modulo 1.  Let $\alpha \in \R$. For all $H \in \N$ for which $\{\alpha, 2 \alpha, \ldots, H \alpha \} \cap \Z = \emptyset$ and all $L \in \N$,
\[\sum_{\ell = 1}^L \big[ \fracpart {\al \ell} \in \openint' \big] \leq L \lambda(\openint) + \frac {2L}{H+1} + 6 \sum_{h=1}^H \frac{1}{h \fracdist {h \al}}.\]
\end{lemma}

\begin{proof}
Note that $\openint'$ is a union of at most two intervals in $[0,1)$. Apply \cite[Theorem 5.5]{harmanbook} to the sequence $(\{\al \ell\})_{\ell=1}^L$ for each of these intervals, using the remarks on pages 130 and 131 in \cite{harmanbook}.
\end{proof}

The following lemma is similar to, but does not follow from the statement of, \cite[Lemma 6.2]{harmanbook} and concerns the number of solutions to the inequality $|qr-ps| \leq L$. Three aspects of this lemma will be most important in its application to the lemma following it: 1) the implicit constant in the conclusion is independent of the length of the interval $J$; 2) the interval $J$ is permitted to contain $0$; and 3) the lack of restriction on $s$ leading to the asymmetry between the conditions on $r/p$ and $s/q$.

\begin{lemma}\label{lem:basicsolutionstoinequality}
Let $p, Q \in \N$, $p$ an odd prime less than $Q$, $L \geq 1$, $\bigQ \subseteq \{Q, \ldots, 2Q-1\} \setminus p \Z$, and $\openint \subseteq \R$ be a non-empty, open interval with length $\lambda(\openint) > p^{-1}$. The number of tuples $(q,r,s) \in \Z^3$ satisfying
\[q \in \bigQ, \ r \big/ p \in \openint, \ |qr-ps| \leq L\]
is
\[ \ll \lambda(\openint)L|\bigQ| + Q (\log Q)^2.\]
\end{lemma}

\begin{proof}
Since $L \geq 1$, by replacing $L$ with $\bothceil L$, we may assume that $L \in \N$. We wish to bound from above the cardinality of the set
\[ \bigS = \big\{ (q,r,s,\ell) \in \Z^4 \ \big | \ q \in \bigQ, \ r \big/ p \in \openint, \ |\ell | \leq L, \ qr-ps=\ell \big\}.\]
For $(q,r,s,\ell) \in \bigS$,
\begin{align}\label{eqn:importantcongruence}r \equiv \overline q \ell \pmod p,\end{align}
where $\overline q$ denotes the positive integer less than $p$ for which $\overline q q \equiv 1 \pmod p$.

If $\lambda(\openint) \geq 1 \big/ 2$, then for fixed $q \in \bigQ$ and $|\ell|\leq L$, the congruence (\ref{eqn:importantcongruence}) together with $r \in p\openint \cap \Z$ implies that there are $\ll \lambda (\openint)$ choices for $r$ for which $(q,r,s,\ell) \in \bigS$. Since $q$, $r$, and $\ell$ determine $s$, we have $|\bigS| \ll \lambda(\openint)L|\bigQ|$.

Suppose $\lambda(\openint) < 1 \big/ 2$. The congruence (\ref{eqn:importantcongruence}) together with $r \big/ p \in \openint$ gives
\[\fracpart {\frac {\overline q \ell}p} = \fracpart {\frac rp} \in \openint',\]
where $J' = \big\{ \{x\} \ \big| \ x \in J \big\} \subseteq [0,1)$ is $J$ reduced modulo 1. It follows that if $(q,r,s,\ell) \in \bigS$, then $(q, \ell) \in \bigS'$, where
\[\bigS' = \big\{ (q,\ell) \in \Z^2 \ \big | \ q \in \bigQ, \ |\ell | \leq L, \ \fracpart {\overline q \ell \big/ p} \in \openint' \big\}.\]
Since $\lambda(\openint) < 1 \big/ 2$, for fixed $q \in \bigQ$ and $|\ell|\leq L$, there is at most 1 choice for $r$ for which there exists an $s$ such that $(q,r,s,\ell) \in \bigS$. Since $q,r,\ell$ determine $s$, the map $(q,r,s,\ell) \mapsto (q,\ell)$ from $\bigS$ into $\bigS'$ is injective, meaning $|\bigS| \leq |\bigS'|$. We proceed by bounding $|\bigS'|$ from above.

Setting $[\textsc{expression}]$ to 1 if \textsc{expression} is true and 0 otherwise,
\begin{align*}
|\bigS'| &= \sum_{q \in \bigQ} \sum_{\ell = -L}^L \left[ \fracpart {\frac {\overline q \ell}p} \in \openint' \right]\\
&\leq \sum_{q \in \bigQ} \left( 1+ \sum_{\ell = 1}^L \left[ \fracpart {\frac {\overline q \ell}p} \in \openint' \right] + \sum_{\ell = 1}^L \left[ \fracpart {\frac {- \overline q \ell}p} \in \openint' \right] \right).
\end{align*}
Since $|\bigQ| \ll Q( \log Q)^2$ and since the bound on the third term follows exactly as the bound on the second, it suffices to show
\begin{align}\label{eqn:targetboundonsecondsummand} \sum_{q \in \bigQ} \sum_{\ell = 1}^L \left[ \fracpart {\frac {\overline q \ell}p} \in \openint' \right] \ll \lambda(\openint)L|\bigQ| + Q (\log Q)^2. \end{align}

Applying Lemma \ref{lem:discrepofnalpha} with $H = \bothfloor {\lambda(\openint)^{-1}}$ and $\alpha = \overline q \big / p$,
\[\sum_{\ell = 1}^L \left[ \fracpart {\frac {\overline q \ell}p} \in \openint' \right] \ll \lambda(\openint) L + \sum_{h=1}^H \frac 1{h \fracdist {h \overline q \big / p}}, \]
where we used $\lambda(\openint) \in (p^{-1},1 / 2)$ to give that $2 \leq H < p$ and $H^{-1} \ll \lambda(\openint)$. This choice of $H$ satisfies the conditions in Lemma \ref{lem:discrepofnalpha}: $h \overline q \big / p \notin \Z$ since $\overline q, h \in \{1, \ldots, p-1\}$. When we sum this expression over $q \in \bigQ$, the first term on the right hand side of (\ref{eqn:targetboundonsecondsummand}) follows immediately. Interchanging the order of summation in the second term, the second term on the right hand side of (\ref{eqn:targetboundonsecondsummand}) would follow from
\begin{align}\label{eqn:twoboundstoshow} \sum_{q \in \bigQ} \frac 1{\fracdist {h \overline q \big / p}} \ll Q \log Q, \text{ and } \sum_{h=1}^H h^{-1} \ll \log Q.\end{align}

The second inequality in (\ref{eqn:twoboundstoshow}) follows from the fact that $\lambda(\openint) > Q^{-1}$. For the first inequality, note that for each $q \in \bigQ$, there exists $v_q \in \{1, \ldots (p-1)/2\}$ for which $\fracdist {h \overline q \big / p} = v_q \big / p$. Since $v_{q} = v_{q'}$ only if $\overline q \equiv \pm \overline {q'} \pmod p$, the map $q \mapsto v_q$ is at most $2(1+ Q \big/ p)$-to-1. It follows that
\[\sum_{q \in \bigQ} \frac 1{\fracdist {h \overline q \big / p}} \ll 2 \left(1+ \frac {Q}p \right)  \sum_{v=1}^{\frac {p-1}2} \frac pv \ll Q \log Q,\]
which completes the proof of (\ref{eqn:twoboundstoshow}) and the proof of the lemma.
\end{proof}



The following lemma is at the heart of Theorem \ref{thm:mainthm} and concerns counting solutions to a perturbation of the inequality $|qr-ps| \leq L$.

\begin{lemma}\label{lem:hardersolutionstoinequality}
Let $p, Q \in \N$, $p$ an odd prime less than $Q$, $L \geq 1$, $\bigQ \subseteq \{Q, \ldots, 2Q-1\} \setminus p \Z$, and $\openint \subseteq \R$ be a non-empty, open interval. Suppose $\big\{\cent_n: \overline{\openint} \to \R\big\}_{n \in \{p\} \cup \bigQ}$ is a collection of $C^1$ functions satisfying
\begin{align}\label{eqn:etainequality}\omega := \frac{1}{10} \min_{n \in \{p\} \cup \bigQ} \left( \lambda(\openint), \frac L{Q \supnormopenint {\cent_n'}}\right) \geq 10 \max_{n \in \{p\} \cup \bigQ} \left( \frac 1p, \frac L{pQ}, \frac {\supnormopenint {\cent_n}}n \right).\end{align}
The number of tuples $(q,r,s) \in \Z^3$ satisfying
\[q \in \bigQ, \ \frac rp, \ \frac sq \in \openint, \ \big|q\big(r-\cent_p(r/p)\big)-p\big(s-\cent_q(s/q)\big)\big| \leq L\]
is
\[\ll \lambda(\openint) \left( L|\bigQ| + \omega^{-1} Q (\log Q)^2 \right).\]
\end{lemma}

\begin{proof}
Since $L \geq 1$, by replacing $L$ with $\bothceil L$, we may assume that $L \in \N$. We wish to bound the cardinality of the set
\[ \bigT = \left\{ (q,r,s) \in \Z^3 \ \big | \ q \in \bigQ, \ \frac rp, \ \frac sq \in \openint, \ \big|q\big(r-\cent_p(r/p)\big)-p\big(s-\cent_q(s/q)\big)\big| \leq L \right\}.\]

Let $\{\openint_k\}_{k \in K}$ be a collection of $|K| = \bothfloor {10 \lambda(\openint) \omega^{-1}}$ open intervals in $J$, each of length $\omega \big/ 5$, covering $\openint$. For each $k \in K$, denote by $\bigT_k$ the set
\[\left\{ (q,r,s) \in \Z^3 \ \big | \ q \in \bigQ, \ \frac rp \in \openint_k, \ \frac sq \in \openint, \ \big|q\big(r-\cent_p(r/p)\big)-p\big(s-\cent_q(s/q)\big)\big| \leq L \right\},\]
and note that since $s \big/ q$ is required only to be in $\openint$, $\bigT \subseteq \cup_{k \in K} \bigT_k$. Therefore, it suffices to prove that for all $k \in K$,
\begin{align}\label{eqn:inequalityforsolutionsonsmallinterval}| \bigT_k | \ll \omega L|\bigQ| + Q (\log Q)^2.\end{align}

Fix $k \in K$, and suppose $(q, r_1, s_1)$, $(q, r_2, s_2) \in \bigT_k$. Putting $r_0 = r_1 - r_2$ and $s_0 = s_1 - s_2$, we will show that $(q, r_0, s_0) \in \bigT_0$, where
\[ \bigT_0 = \big\{ (q,r_0,s_0) \in \Z^3 \ \big | \ q \in \bigQ, \ r_0 \big/ p \in (-\omega,\omega), \ \big|qr_0-ps_0\big| \leq 3L \big\}.\]
To see that $(q, r_0, s_0) \in \bigT_0$, note first that since $\openint_k$ is an open interval of length $\omega \big/ 5$, $\big| r_1 / p - r_2 / p \big| < \omega \big/ 5$, meaning $r_0 \big/ p \in (-\omega,\omega)$. Moreover, $\big| s_1 / q - s_2 / q \big| < \omega$: for $i=1,2$,
\[\left| \frac{r_i-\cent_p(r_i/p)}{p}-\frac{s_i-\cent_q(s_i/q)}{q} \right| = \left| \frac{r_i}{p}-\frac{s_i}{q} - \left( \frac{\cent_p(r_i/p)}{p}-\frac{\cent_q(s_i/q)}{q}\right) \right| \leq \frac {L}{pQ},\]
whereby it follows from several applications of the triangle inequality and (\ref{eqn:etainequality}) that
\begin{align*}
\left| \frac{s_1}{q}-\frac{s_2}{q} \right| &\leq \left| \frac{s_1}{q} - \frac{r_1}{p} \right| + \left| \frac{r_1}{p}-\frac{r_2}{p} \right| + \left| \frac{r_2}{p}-\frac{s_2}{q} \right|\\
&< 2 \left( \frac {L}{pQ} + \frac {\supnormopenint {\cent_p}}{p} + \frac {\supnormopenint {\cent_q}}{q} \right) + \frac{\omega}5 \leq 2 \frac{3\omega}{10} + \frac \omega 5 < \omega.
\end{align*}
Finally, by the triangle inequality, the MVT, and (\ref{eqn:etainequality}),
\begin{align*}
\big|qr_0-ps_0\big| & \begin{aligned} = \big|qr_1 - ps_1 - (qr_2 - ps_2)\big| \end{aligned} \\
& \begin{aligned} \leq & \ \big|qr_1 - ps_1 - \big(q \cent_p(r_1/p) - p\cent_q(s_1/q) \big) \big| + \\ & \qquad \big| \big(q \cent_p(r_1/p) - p\cent_q(s_1/q) \big) - \big(q \cent_p(r_2/p) - p\cent_q(s_2/q) \big) \big| + \\  & \qquad  \big| \big(q \cent_p(r_2/p) - p\cent_q(s_2/q) \big) - (qr_2 - ps_2) \big| \end{aligned}\\
& \begin{aligned} = & \ \big|q\big(r_1-\cent_p(r_1/p)\big)-p\big(s_1-\cent_q(s_1/q)\big)\big| + \\ & \qquad \big| q\big(\cent_p(r_1/p) - \cent_p(r_2/p) \big) - p\big(\cent_q(s_1/q) - \cent_q(s_2/q) \big) \big| + \\  & \qquad  \big|q\big(r_2-\cent_p(r_2/p)\big)-p\big(s_2-\cent_q(s_2/q)\big)\big| \end{aligned}\\
& \begin{aligned} \leq L + q \left| \frac {r_1}p - \frac {r_2}p \right| \supnormopenint{\cent_p'} + p \left| \frac {s_1}q - \frac {s_2}q \right| \supnormopenint{\cent_q'} + L \end{aligned}\\
& \begin{aligned} \leq 2L + 3 Q \omega \max_{n \in \{p\} \cup \bigQ} \supnormopenint {\cent_n'} \leq 3L. \end{aligned}
\end{align*}
This shows that $(q, r_0, s_0) \in \bigT_0$.

In each (non-empty) fiber of the map $\bigT_k \to \Z$ defined by $(q,r,s) \mapsto q$, fix one point $(q, r_q, s_q) \in \bigT_k$. The map $\bigT_k \to \Z^3$ defined by $(q,r,s) \mapsto (q,r-r_q,s-s_q)$ is injective, and by the work above, its image lies in $\bigT_0$. It follows that $|\bigT_k| \leq |\bigT_0|$, meaning that in order to show (\ref{eqn:inequalityforsolutionsonsmallinterval}), it suffices to show
\[ | \bigT_0 | \ll \omega L|\bigQ| + Q (\log Q)^2. \]
This inequality follows by applying Lemma \ref{lem:basicsolutionstoinequality} with $p$, $Q$, and $\bigQ$ as they are, $3L$ as $L$, and $(-\omega,\omega)$ as $\openint$, noting that $2 \omega > p^{-1}$ follows from (\ref{eqn:etainequality}).
\end{proof}

\section{Proof of main theorem}\label{sec:proofofmainthm}

In this section, we state and prove the full version of Theorem \ref{thm:mainthm}. In order to properly formulate the equidistribution condition, we need the following definition.

\begin{definition}
For $n \in \N$, let $S_n \subseteq \Z$ be finite and $F_n: S_n \to \R$ be a finite sequence of elements of $\R$ indexed by $S_n$. The sequence $(F_n )_{n \in \N}$ \emph{equidistributes modulo 1} if for all intervals $I \subseteq [0,1)$,
\[\lim_{n \to \infty}\frac{1}{|S_n|} \big| \{ m \in S_n \ | \ \{F_n(m)\} \in I \} \big| = \lambda(I).\]
We will frequently denote a finite sequence $F$ indexed by $S \subseteq \Z$ by $\big(F(m)\big)_{m \in S}$.
\end{definition}

To ease notation in the statement and proof of the main theorem, any sum indexed over $p$ or $q$ will be understood to be a sum over prime numbers. For $J \subseteq \R$ and $n \in \N$, the set $nJ$ is $\{n j \ | \ j \in J\}$.

\begin{namedthm*}{Theorem \ref{thm:mainthm}}
Let $J \subseteq \R$ be a non-empty, open interval. Let $(\psi_n)_{n \in \N}$ be a sequence in $(0,1/10)$ satisfying:
\begin{enumerate}
\item[(A1)] $\displaystyle \sup_{n \in \N} \left(\sum_{\ell =1}^\infty \frac{\cize_{2^\ell n}\ell^2}{\cize_n} \right) < \infty$;
\item[(A2)] the sequence $\displaystyle \left(\cize_n \big/ n\right)_{n \in \N}$ is non-increasing;
\item[(A3)] there exists $c > 1$ such that for all $n \in \N$, $\displaystyle \cize_{n} < c \cize_{2n}$; and
\item[(A4)] $\displaystyle \sum_{p = 2}^\infty \cize_p = \infty$.
\end{enumerate}
Let $\big(\cent_n: \overline{\openint} \to \R \big)_{n\in\N}$ be a sequence of $C^1$ functions satisfying:
\begin{enumerate}
\item[(B1)] $\displaystyle \sup_{n \in \N} \supnormopenint {\cent_n} < \infty$;
\item[(B2)] the sequence $\displaystyle \big(\supnormopenint {\cent_n'}\big)_{n\in\N}$  is eventually non-increasing;
\item[(B3)] $\displaystyle\lim_{N \to \infty} \frac{\sum_{p = 2}^N \max(\cize_p, \supnormopenint {\cent_p'}) (\log p)^2}{\left( \sum_{p = 2}^N \cize_p \right)^2} = 0$; and
\item[(B4)] $\displaystyle \lim_{n \to \infty} \frac{\supnormopenint {\cent_n'}}{ n \cize_n}= 0$.
\end{enumerate}
Let $(\rho_n: \overline{J} \to \R)_{n \in \N}$ be a sequence of $C^1$ functions satisfying: for all proper subintervals $J' \subseteq J$,
\begin{enumerate}
\item[(C1)] the sequence $\displaystyle \left(\left( \rho_n\left( \frac {m - \varphi_n(m/n)}{n} \right) \right)_{m \in n J' \cap \Z} \right)_{n \in \N}$ equidistributes modulo $1$; and
\item[(C2)] $\displaystyle \lim_{n \to \infty} \frac{\cize_n \| \rho_n'\|_{\overline{J'},\infty}}{n} = 0$.
\end{enumerate}
Let $I \subseteq [0,1)$ be a non-empty interval. For Lebesgue-a.e. $\theta \in \openint$, the system
\begin{align}\label{eqn:systemnew}\begin{dcases} \big\| \theta n + \cent_n(\theta) \big\| \leq \cize_n \\ \fracpart {\rho_n(\theta)} \in I \end{dcases}\end{align}
is solvable.
\end{namedthm*}

Some discussion about the long list of assumptions is in order before the proof.  The A conditions place restrictions on the accuracy of the approximation $\psi_n$. Roughly speaking, conditions (A1) and (A2) say that $\psi_n$ decreases sufficiently rapidly, while (A3) and (A4) say that $\psi_n$ does not decrease too rapidly.  It is quick to check that (A2) is a weaker monotonicity assumption than the one that appears in Khintchine's theorem (that $n \psi_n$ is non-increasing).

It is easy to check that when $0 < \sigma < 1$, the sequence $\cize_n = 1/n^\sigma$ satisfies all of the A conditions and (B3) with $\psi_p$ as the maximum.  The main theorem in this special case is reformulated in the introduction.

As was discussed in the introduction, if $\|\varphi_n\|_{\overline{J},\infty} \ll \cize_n$, then the perturbation in the inequality in (\ref{eqn:systemnew}) can be removed without any harm by rescaling $\cize_n$. Thus, while it is not explicitly required, the theorem is most interesting when $\|\varphi_n\|_{\overline{J},\infty}$ exceeds $\cize_n$.

The theorem is such that if $I = [0,1)$, then the second condition in  (\ref{eqn:systemnew}) is automatically satisfied. In this case, there is no need to define the $\rho_n$'s or verify that the C conditions hold.

\begin{proof}[Proof of Theorem \ref{thm:mainthm}]
For brevity, we will suppress the dependence on $J$, $(\psi_n)_{n \in \N}$, $(\cent_n)_{n \in \N}$, $(\rho_n)_{n \in \N}$, and $I$ in the asymptotic notation appearing in the proof.

Write $J = (j_1,j_2)$, and let $\Theta \subseteq \openint$ be the set of those $\theta$ satisfying the conclusion of the theorem. To show that $\Theta$ is of full measure, it suffices by Lemma \ref{lem:densitylemma} to show that there exists a $\de > 0$ such that for all $j_1 < \intone <  \inttwo < j_2$,
\begin{align}\label{eqn:biggerthandelta}\lambda \left(\Theta \cap (\intone,\inttwo)\right) \geq \de (\inttwo - \intone).\end{align}
To this end, fix $j_1 < \intone <  \inttwo < j_2$. In what follows, the phrase ``for all sufficiently large $n$'' means ``for all $n \geq n_0$,'' where $n_0 \in \N$ may depend on any of the quantities and sequences introduced so far, including $\intone$ and $\inttwo$.

For $n \in \N$, define 
\begin{align*}
E_n &= \left\{ \theta \in (\intone,\inttwo) \ \middle | \ \fracdist {\theta n + \cent_n(\theta)} \leq \cize_n \right\},\\
F_n &= \left\{ \theta \in (\intone,\inttwo) \ \middle | \ \fracpart {\rho_n(\theta)} \in I \right\}.
\end{align*}
Put $G_n = E_n \cap F_n$, and note that $\limsup_{n \to\infty}G_n = \Theta \cap (\intone, \inttwo)$. Therefore, in order to show (\ref{eqn:biggerthandelta}), it suffices to prove that there exists a $\de > 0$, independent of $\intone, \inttwo$, for which
\begin{align}\label{eqn:mainboundonlimsup}\lambda \left( \limsup_{\substack{p \to \infty\\ p \text{ prime}}} G_p \right) \geq \de (\inttwo - \intone).\end{align}
Passing to primes here makes parts of the later argument technically easier.

To prove (\ref{eqn:mainboundonlimsup}), it suffices by Lemma \ref{lem:advancedBC} to prove that
\begin{align}\label{eqn:Gpsarebigenough} \sum_{p = 2}^\infty \lambda(G_p) = \infty\end{align}
and that there exists a $\de > 0$ independent of $\intone, \inttwo$ for which
\begin{align}\label{eqn:Gpsareindependent} \limsup_{N \to \infty} \left( \sum_{p = 2}^N \lambda(G_p) \right)^2 \left( \sum_{p,q = 2}^N \lambda(G_p \cap G_q) \right)^{-1} \geq \de (\inttwo - \intone).\end{align}

First we show (\ref{eqn:Gpsarebigenough}) using Lemma \ref{lem:equid}. Fix $0 < \eta < \min \big( \intone-j_1, j_2-\inttwo, (\inttwo-\intone)/3 \big)$. For $n \in \N$, let
\begin{align*}
\begin{split} S_n &= \big\{ m \in \Z \ \big| \ \intone + \eta < m / n < \inttwo-\eta \big\},\\
T_n &= \big\{ m \in \Z \ \big| \ \intone - \eta < m / n < \inttwo+\eta \big\}, \end{split}
\end{align*}
and note that by the bounds on $\eta$, for $n$ sufficiently large,
\begin{align}\label{eqn:boundsonsets} (\inttwo-\intone)n \ll |S_n| < |T_n| \ll (\inttwo-\intone)n. \end{align}
To approximate the set $E_n$ by a union of intervals, define
\begin{align*}
e_{n,m} & = \frac {m-\varphi_n(m \big / n)}{n},\\
E_{n,m} &= e_{n,m} + \frac 12 \left[-\frac{\cize_n}{n}, \frac{\cize_n}{n} \right],\\
E_{n,m}' &= e_{n,m} + 2 \left[-\frac{\cize_n}{n}, \frac{\cize_n}{n} \right].
\end{align*}
It follows from the fact that $\cize_n < 1/10$, the assumptions in (B1) and (B4), the definition of $E_n$, and estimates with the MVT that for $n$ sufficiently large, the $E_{n,m}'$'s are disjoint and
\begin{align}\label{eqn:Enisaunionofintervals}\bigcup_{m \in S_n} E_{n,m} \subseteq E_n \subseteq \bigcup_{m \in T_n} E_{n,m}'.\end{align}
This shows $\lambda(G_n) \leq \lambda(E_n) \ll (\inttwo - \intone)\cize_n$.

Let $I_0 \subseteq I$ be the middle third sub-interval of $I$.

\begin{claim}\label{lem:inInaughtimpliesinI}
For $n$ sufficiently large and $m \in S_n$, if $\fracpart {\rho_n(e_{n,m}) } \in I_0$, then $E_{n,m} \subseteq F_n$.
\end{claim}

\begin{proof}
Let $\theta \in E_{n,m}$. By the MVT, we see that for some $\xi$ between $\theta$ and $e_{n,m}$,
\[\big|\rho_n(\theta) -  \rho_n(e_{n,m}) \big| \leq \big| \theta - e_{n,m} \big| \big| \rho_n'(\xi) \big| \ll \frac{\cize_n}n  \| \rho_n'\|_{\overline{J},\infty}.\]
By the assumptions in (C2), the right hand side tends to zero as $n \to \infty$, so for $n$ sufficiently large,
\[\big|\rho_n(\theta) -  \rho_n(e_{n,m}) \big| < \frac{\lambda(I)}{3}.\]
Therefore, if $\rho_n(e_{n,m}) \in I_0$, then for all $\theta \in E_{n,m}$, $\rho_n(\theta) \in I$. This implies that $E_{n,m} \subseteq F_n$.
\end{proof}

By the equidistribution assumption in (C2), for $n$ sufficiently large,
\[\frac{\left| \left\{ m \in S_n \ \middle| \ \fracpart {\rho_n(e_{n,m})} \in I_0 \right\} \right|}{|S_n|} \geq \frac{\lambda(I_0)}{2}.\]
Combining this with (\ref{eqn:boundsonsets}) and Claim \ref{lem:inInaughtimpliesinI}, there are $\gg (\inttwo-\intone)n$ integers $m \in S_n$ for which $E_{n,m} \subseteq F_n$. It follows by the disjointness of the intervals $E_{n,m}$ that for $n$ sufficiently large,
\begin{align}\label{eqn:gnisbig}\lambda(G_n) \gg (\inttwo - \intone)n \frac{\cize_n}{n} = (\inttwo - \intone) \cize_n.\end{align}
Now (\ref{eqn:Gpsarebigenough}) follows by the assumption in (A4).


Now we show (\ref{eqn:Gpsareindependent}) by estimating the ``overlaps'' between the $G_p$'s. First we show that it suffices to prove that for all sufficiently large primes $p$ (potentially depending on $\intone$, $\inttwo$) and for all $N > p$,
\begin{align}\label{eqn:sufficientupperboundcondition}
\sum_{q > p}^N \lambda (E_p \cap E_q) \ll (\inttwo-\intone) \sum_{q > p}^N  \cize_p \cize_q  + \max \left(\cize_p, \supnormopenint {\cent_p'} \right) \left( \log p \right)^2.
\end{align}
Indeed, suppose (\ref{eqn:sufficientupperboundcondition}) holds for all primes $p$ greater than some sufficiently large $p_0 \in \N$. Using the trivial bound $\lambda(G_p \cap G_q) \leq \lambda(E_p \cap E_q)$, it follows that
\begin{align*}
\sum_{p, q = 2}^N \lambda(G_p &\cap G_q) \leq 2\left( \sum_{\substack{p \geq p_0 \\ q > p}}^N \lambda(G_p \cap G_q) + \sum_{\substack{p < p_0 \\ q > p}}^N \lambda(G_q) \right) + \sum_{q = 2}^N \lambda(G_q) \\
&\ll (\inttwo-\intone) \sum_{\substack{p \geq p_0 \\ q > p}}^N \cize_p \cize_q  + \sum_{p\geq p_0}^N \max \left(\cize_p, \supnormopenint {\cent_p'} \right) (\log p)^2 + \sum_{\substack{p < p_0 \\ q \geq p}}^N \lambda(G_q)\\
&\ll \frac 1{\inttwo-\intone} \sum_{p, q = 2}^N \lambda(G_p) \lambda(G_q)  + p_0 \sum_{p = 2}^N \max \left(\cize_p, \supnormopenint {\cent_p'} \right) (\log p)^2,
\end{align*}
where the last line follows from (\ref{eqn:gnisbig}) and the fact that $\lambda (G_q) \ll \cize_q$. The inequality in (\ref{eqn:Gpsareindependent}) follows from this estimate because (\ref{eqn:gnisbig}) and (B3) imply that
\[\lim_{N \to \infty} \frac{p_0\sum_{p = 2}^N \max \left(\cize_p, \supnormopenint {\cent_p'} \right) (\log p)^2}{\left( \sum_{p = 2}^N \lambda(G_p)\right)^2} = 0.\]

To show (\ref{eqn:sufficientupperboundcondition}), note that by (\ref{eqn:Enisaunionofintervals}), the set $E_p$ is covered by a union of intervals $E_{p,r}'$, each of length $4 \cize_p \big / p$. If $p < q$ and $E_{p,r}' \cap E_{q,s}' \neq \emptyset$, then by estimating the distance between the midpoints of the intervals and using that $\cize_n/n$ is non-increasing (A2),
\begin{align}
\label{eqn:overlapinequality}\left| q \left( r - \cent_p (r /p)\right) - p \left( s - \cent_q (s/q)\right) \right| &\leq pq \ 4\max \left( \frac{\cize_p}p, \frac{\cize_q}q \right) = 4 q \cize_p,\\
\notag \lambda \left( E_{p,r}' \cap E_{q,s}' \right) &\leq 4 \min \left( \frac {\cize_p}{p}, \frac {\cize_q}{q} \right) = 4 \frac {\cize_q}{q}.
\end{align}

The left hand side of (\ref{eqn:sufficientupperboundcondition}) is then
\begin{align*}
\sum_{q > p}^N \lambda(E_p \cap E_q) \leq \sum_{q > p}^N \sum_{\substack{r \in T_p \\ s \in T_q}} \lambda \left( E_{p,r}' \cap E_{q,s}' \right) \ll \sum_{q > p}^N \frac {\cize_q}{q} \sum_{\substack{r \in T_p, \ s \in T_q \\ \text{(\ref{eqn:overlapinequality}) holds}}} 1.
\end{align*}
Now (\ref{eqn:sufficientupperboundcondition}) will follow from Lemma \ref{lem:hardersolutionstoinequality} by partitioning $\{p, \ldots, N\}$ dyadically. Indeed, using (A2), the right hand side of the previous expression is
\begin{align}\label{eqn:targetinequality}\sum_{\ell = 0}^\infty \sum_{\substack{2^\ell p < q < 2^{\ell+1} p \\ q \leq N}} \frac {\cize_q}{q} \sum_{\substack{r \in T_p, \ s \in T_q \\ (q,r,s) \text{ satisfies (\ref{eqn:overlapinequality})}}} 1 \leq \sum_{\ell = 0}^\infty \frac {\cize_{2^\ell p}}{2^\ell p} \sum_{\substack{q \in \bigQ_\ell, \ r \in T_p, \ s \in T_q \\ (q,r,s) \text{ satisfies (\ref{eqn:overlapinequalitytwo})}}} 1,\end{align}
where $\bigQ_\ell$ are those primes $q$ for which $2^\ell p < q < \min \left(2^{\ell+1} p, N+1 \right)$, and
\begin{align}
\label{eqn:overlapinequalitytwo}\big| q \left( r - \cent_p (r /p)\right) - p \left( s - \cent_q (s/q)\right) \big| &\leq 4\cdot 2^{\ell+1} p \cize_p.
\end{align}
For each $\ell$, we apply Lemma \ref{lem:hardersolutionstoinequality} with $p$ as it is, $2^\ell p$ as $Q$, $2^{\ell+3}p \cize_p$ as $L$, $\bigQ_\ell$ as $\bigQ$, $(\intone - \eta, \inttwo + \eta)$ as $\openint$, $\varphi_n$ restricted to $[\intone - \eta, \inttwo + \eta]$ as $\varphi_n$, and
\[\omega_p := \frac{1}{10} \min\left( \inttwo-\intone+2\eta, \frac{2^3 \cize_p}{\|\cent_p'\|_{[\intone - \eta, \inttwo + \eta],\infty}} \right) \text{ as } \omega.\]
The conditions for the lemma are met for $p$ sufficiently large by the assumptions in (B1), (B2), and (B4), and the conclusion is that the right-most summand in (\ref{eqn:targetinequality}) is
\begin{align}\label{eqn:itermediatestepininequality}\ll (\inttwo-\intone  +2\eta) \left( 2^{\ell+3}p \cize_p \sum_{\substack{2^\ell p < q < 2^{\ell+1} p \\ q \leq N}} 1 + \omega_p^{-1} 2^\ell p (\log (2^\ell p))^2 \right).\end{align}

To bound this further, we derive two inequalities from the assumptions in (A1) and (A3). It follows from (A3) that when $2^\ell p < q < 2^{\ell + 1} p$,
\[\cize_{2^\ell p} < c \cize_{2^{\ell + 1}p} < c 2^{\ell + 1}p \cize_q / q < 2 c \cize_q.\]
Additionally, it follows from (A1) that
\[\sum_{\ell = 0}^\infty \cize_{2^\ell p} \left( \log \left( 2^\ell p \right) \right)^2 \ll \cize_p \left(\log p \right)^2,\]
where the implied constant in the asymptotic notation is independent of $p$. With these inequalities and (\ref{eqn:itermediatestepininequality}), the right hand side of (\ref{eqn:targetinequality}) is bounded from above by
\begin{align*}
&\ll (\inttwo-\intone) \left( \cize_p \sum_{\ell = 0}^\infty \sum_{\substack{2^\ell p < q < 2^{\ell+1} p \\ q \leq N}} \cize_{2^\ell p} +  \omega_p^{-1} \sum_{\ell = 0}^\infty \cize_{2^\ell p} \left( \log \left( 2^\ell p \right) \right)^2 \right)\\
&\ll 2c (\inttwo-\intone) \cize_p \sum_{\ell = 0}^\infty \sum_{\substack{2^\ell p < q < 2^{\ell+1} p \\ q \leq N}} \cize_q + \cize_p \omega_p^{-1} \left( \log  p  \right)^2 \\
&\ll (\inttwo-\intone) \sum_{q > p}^N  \cize_p \cize_q  + \max \left(\cize_p, \supnormopenint {\cent_p'} \right) \left( \log p \right)^2.
\end{align*}
This shows (\ref{eqn:sufficientupperboundcondition}), completing the proof of the theorem.
\end{proof}

\section{Solutions to linear equations in P-S sequences}\label{sec:reducttodiophantineapprox}

In this section, we prove Theorem \ref{thm:mainapplication} and Corollary \ref{cor:maincortwo} with the help of Theorem \ref{thm:mainthm}. The first step in the proof of Theorem \ref{thm:mainapplication} is a reduction of the problem to one in Diophantine approximation. This setting arises naturally when solving the equation $a \intpart {n^\al} + b = \intpart{m^\al}$ for $n$.

\begin{theorem}\label{thm:firstDAform}
Let $0 < a < 1$, $I \subseteq [0,1)$ be a non-empty interval, and $\ka, c, \ga \in \R$ with $c > 0$ and $\gamma \neq 0$. For Lebesgue-a.e. $\al > 1$, the system
\begin{align}\label{eqn:systemone}\begin{dcases} \fracdist {na^{1/\al} + \frac{\ka a^{1/\al}}{\al n^{\al - 1}} } \leq \frac{c}{n^{\al-1}} \\ \fracpart {\ga n^\al} \in I \end{dcases}\end{align}
is solvable or unsolvable in $\N$ according as $\al < 2$ or $\al > 2$.
\end{theorem}

\begin{proof}[Proof of Theorem \ref{thm:mainapplication} assuming Theorem \ref{thm:firstDAform}]
Since $a \neq 1$, by interchanging $x$ and $y$ if necessary, we may assume that $|a| < 1$. By assumption, $a \neq 0$ and (\ref{eqn:main}) is solvable in $\N$, and this implies that
\begin{align*}a, b \in \Q, \ a = \frac{a_1}{a_2} > 0, \ a_1, a_2 \in \N, \ (a_1, a_2) = 1, \ \text{and } a_2b \in \Z.\end{align*}
Let $d \in \{0,\ldots,a_2 - 1\}$ be such that $d a_1 \equiv -b a_2 \pmod{a_2}$, and note that for all $r \in \R$,
\[a \intpart {r} + b\in \Z \iff \intpart{r} \equiv d \pmod {a_2} \iff \fracpart {\frac {r}{a_2}} \in \left[ \frac{d}{a_2}, \frac{d+1}{a_2} \right).\]
It follows that
\begin{align}
\notag a \intpart {n^\al} + b\in \psa \iff & \exists \ k \in \N, \ a \intpart {n^\al} + b = \intpart{k^\al}\\
\notag \iff & \begin{cases} a \intpart {n^\al} + b \in \Z \ \text{ and } \\ \exists \ k \in \N, \ a \intpart {n^\al} + b \leq k^\al < a \intpart {n^\al} + b + 1 \end{cases} \\
\label{eqn:suffandnecconditions}\iff & \fracpart {\frac {n^\al}{a_2}} \in \left[ \frac{d}{a_2}, \frac{d+1}{a_2} \right) \ \text{ and } J_n \cap \N \neq \emptyset,
\end{align}
where, by writing $\intpart {n^\al} = n^\al - \fracpart{n^\al}$ and applying the Mean Value Theorem (MVT) twice,
\begin{align*}J_n &= \left[ \left(a \intpart {n^\al} + b\right)^{1/ \al},\left(a \intpart {n^\al} + b + 1\right)^{1/ \al} \right) = a^{1/ \al}n + U_n + [L_n,R_n),\\
U_n &= \frac b{\al} u_n^{-1+1/\al}, \ \ u_n \text{ between } an^\al \text{ and } an^\al + b,\\
L_n &= -\frac a\al \fracpart {n^\al} l_n^{-1 + 1/\al}, \ \ l_n \text{ between } a \intpart{n^\al} +b \text{ and } a n^\al + b,\\
R_n &= \frac a\al \left( \frac {1}a - \fracpart {n^\al} \right) r_n^{-1 + 1/\al}, \ \ r_n \text{ between } a n^\al + b \text{ and } a \intpart {n^\al} + b+1.\end{align*}
Note that $J_n, U_n, u_n, L_n, l_n, R_n, $ and $r_n$ all depend on $\al$. This shows so far that (\ref{eqn:main}) is solvable in $\psa$ if and only if the system in (\ref{eqn:suffandnecconditions}) is solvable in $\N$.

We proceed by showing that solutions to (\ref{eqn:systemone}) yield solutions to (\ref{eqn:suffandnecconditions}) and vice versa when $I$, $\ka$, $c$, and $\ga$ are chosen appropriately. To this end, for $i=1,2$, let
\begin{align*}A &= \{\al > 1 \ | \ (\ref{eqn:suffandnecconditions}) \text{ is solvable in } \N \},\\
B_i &= \{\al > 1 \ | \ (\ref{eqn:systemone}) \text{ is solvable in } \N \text{ for } a, I_i, \ka_i, c_i, \ga_i \}.\end{align*}
To prove Theorem \ref{thm:mainapplication}, it suffices by Theorem \ref{thm:firstDAform} to find $I_i, \ka_i, c_i, \ga_i$, $i=1,2$, for which
\begin{align}\label{eqn:containmenttoprove}B_1 \cap (1,2) \subseteq A \subseteq B_2.\end{align}

We begin with the first containment in (\ref{eqn:containmenttoprove}). Let $I' = [1 \big / 3, 2 \big / 3]$, $I_1 = d/a_2 + I'/a_2$, $\ga_1 = 1/a_2$, $\ka_1 = b \big/ a$, and $c_1$ be a constant depending only on $a$ to be specified momentarily.

Suppose $\al \in B_1 \cap (1,2)$ and that $n$ is a solution to (\ref{eqn:systemone}); we will show that if $n$ is sufficiently large, then it solves the system in (\ref{eqn:suffandnecconditions}). By (\ref{eqn:systemone}),
\[\fracpart {\frac {n^\al}{a_2}} = \fracpart {\ga_1 n^\al} \in I_1 \subseteq \left[ \frac d{a_2}, \frac {d+1}{a_2} \right),\]
which is the first condition in (\ref{eqn:suffandnecconditions}). This also implies that $\fracpart {n^\al} \in I'$, which when combined with the fact that $1 \big / a > 1$, means
\[\fracpart {n^\al} > \frac 13 \quad \text{ and } \quad \frac 1a - \fracpart {n^\al} > \frac 13.\]
Combining these inequalities with the fact that $\al \in (1,2)$ and, for $n$ sufficiently large, $a n^\al + b + 1 \leq 2 a n^\al$, we get
\begin{align*}
\begin{split}
-L_n &= \frac a\al \fracpart {n^\al} l_n^{-1 + 1/\al} \gg_{a} \frac 1{n^{\al-1}},\\
R_n &= \frac a\al \left( \frac {1}a - \fracpart {n^\al} \right) r_n^{-1 + 1/\al} \gg_{a} \frac 1{n^{\al-1}}.
\end{split}
\end{align*}
Let $c_1$ be a third of the minimum of the constants implicit in the previous two expressions. With this choice, $J_n$ contains an open interval centered at $a^{1/\al}n + U_n$ of length $6 c_1 \big / n^{\al - 1}$. By the triangle inequality and an application of the MVT, for $n$ sufficiently large,
\begin{align}\label{eqn:triangleandmvt}\left| \fracdist {a^{1/\al}n + U_n} - \fracdist {a^{1/\al}n + \frac{\ka_1 a^{1/\al}} {\al  n^{\al-1}}} \right| \leq \left| U_n -\frac{\ka_1 a^{1/\al}} {\al  n^{\al-1}} \right| \leq \frac{c_1}{n^{\al-1}},\end{align}
from which it follows by (\ref{eqn:systemone}) that $\fracdist {a^{1/\al}n + U_n} \leq 2c_1 \big/ n^{\al-1}$. This shows that $J_n$ contains the nearest integer to $a^{1/\al}n + U_n$; in particular, $J_n \cap \N \neq \emptyset$, so $n$ solves (\ref{eqn:suffandnecconditions}).

The second containment in (\ref{eqn:containmenttoprove}) is handled similarly. Let $I_2 = [0,1)$, $\ga_2 = 1$, $\ka_2 = b \big / a$, and $c_2$ be a constant depending only on $a$ to be specified momentarily. Suppose that $\al \in A$ and $n$ solves (\ref{eqn:suffandnecconditions}); we will show that $n$ satisfies (\ref{eqn:systemone}).  The second condition in (\ref{eqn:systemone}) is satisfied automatically by our choice of $I_2$. For $n$ sufficiently large, $a \intpart {n^\al} + b \geq a n^\al \big / 2$, whereby $|L_n|, |R_n| \leq c_2 \big/ 2 {n^{\al-1}}$, where $c_2$ is chosen (depending only on $a$) to satisfy both inequalities. Since $J_n$ contains an integer, it must be that $\fracdist{ a^{1/\al}n + U_n} \leq c_2 \big/ 2 n^{\al-1}$. It follows by the triangle inequality and the MVT just as in (\ref{eqn:triangleandmvt}) with an upper bound of  $c_2 \big/ 2 n^{\al-1}$ that $n$ satisfies the first condition in (\ref{eqn:systemone}).
\end{proof}

To prove Theorem \ref{thm:firstDAform}, we first change variables under $\varchange ax = (\log_ax)^{-1}$ to arrive at the equivalent Theorem \ref{thm:firstDAform}$'$.  Note that when $0 < a < 1$, the function $\varchangeempty a$ is increasing and infinitely differentiable on $(a,1)$. Proof of the equivalence of these two theorems is a routine exercise using the fact that $\varchangeempty a$ on $(a,1)$ is non-singular with respect to the Lebesgue measure.

\begin{namedthm*}{Theorem \ref{thm:firstDAform}$'$}
Let $0 < a < 1$, $I \subseteq [0,1)$ be a non-empty interval, and $\ka, c, \ga \in \R$ with $c > 0$ and $\gamma \neq 0$. For Lebesgue-a.e. $a < \theta < 1$, the system 
\begin{align}\label{eqn:systemtwo}\begin{dcases} \fracdist {n\theta + \frac{\ka \theta}{\varchange a \theta n^{\varchange a\theta - 1}}} \leq \frac{c}{n^{\varchange a\theta -1}} \\ \fracpart {\ga n^{\varchange a\theta }} \in I \end{dcases}\end{align}
is solvable or unsolvable in $\N$ according as $\theta < \sqrt a$ or $\theta > \sqrt a$.
\end{namedthm*}

This change of variables reveals the form of Theorem \ref{thm:firstDAform} as a perturbation of the rotation considered in Khintchine's theorem with a twist.  We cannot directly apply Theorem \ref{thm:mainthm} because, as it stands now, the accuracy of the approximation $\psi_n$ in (\ref{eqn:systemtwo}) is a function of the variable $\theta$. We are able to eliminate this unpleasant feature by breaking the interval $(a,\sqrt{a})$ up into suitable subintervals and replacing the upper bound with a worst case bound on that subinterval.

The following equidistribution lemma will help us verify condition (C1) when applying Theorem \ref{thm:mainthm} to prove Theorem \ref{thm:firstDAform}$'$.

\begin{lemma}\label{lem:equid}
Let $0 < a < 1$, $\gamma \in \R \setminus \{0\}$, and $\openint = (j_1,j_2)$ be a non-empty, open interval with closure $\overline{\openint} \subseteq (a,\sqrt{a})$. Let $\big(\cent_n: \overline{\openint} \to \R \big)_{n \in \N}$ be a sequence of $C^3$ functions such that for all $i \in \{0,1,2,3\}$, $\sup_{n\in\N}\supnormopenint {\cent_n^{(i)}} < \infty$. The sequence
\[\left( \left( \gamma n^{\varchange a{m/n - \varphi_n(m/n)/n}} \right)_{m \in nJ \cap \Z} \right)_{n \in \N}\]
equidistributes modulo $1$.
\end{lemma}

\begin{proof}
Let $N_n = \big| \{ m \in \Z \ | \ m/n \in \openint \} \big|$, and note that $N_n \big/ \big( n\lambda(\openint)\big) \to 1$ as $n \to \infty$. For $n, h \in \N$, let
\begin{gather*}
f_n(x) = \frac 1n \left( x+\bothfloor {j_1 n} - \cent_n \left( \frac {x+\bothfloor {j_1 n}}{n}\right)\right),\\
g_n(x) = \gamma n^{\varchange a {f_n(x)}}, \qquad g_{n,h}(x) = g_n(x+h) - g_n(x).
\end{gather*}
Note that $f_n$ and $g_n$ are $C^3$ functions on $[1,N_n]$ and that $g_{n,h}$ is a $C^3$ function on $[1,N_n-h]$. Because the sequences of supremum norms of the derivatives of the $\cent_n$'s are bounded, we see that for $n$ sufficiently large, for all $x \in [1,N_n]$, $|f_n'(x) - n^{-1}| \ll n^{-2}$ while $|f_n^{(i)}(x)| \ll n^{-(i+1)}$ for $i \in \{2,3\}$. 

Using this notation, we must show
\begin{align}\label{eqn:equidistributiontoshow}\frac{1}{N_n} \sum_{i=1}^{N_n} \de_{\fracpart {g_n(i)}} \longrightarrow \restr{\lambda}{[0,1)} \quad \text{ as } n \to \infty,\end{align}
where $\de_{x}$ denotes the unit point mass at the point $x \in [0,1)$ and convergence of probability measures is meant in the weak-$\ast$ topology.

Since $\varchangeempty a$ is strictly increasing on $(a,\sqrt{a})$ and $\overline\openint \subseteq (a,\sqrt{a})$, we can fix $\sigma_1$, $\sigma_2$ such that for all $x \in J$,
\[1 < \sigma_1 < \varchange a{x} < \sigma_2 < 2.\]
To handle the exponential sum estimates that follow, we will show that there exist positive constants $C_1$ and $C_2$ (depending on $a$ and $\gamma$) such that for all $h \in \N$, all sufficiently large $n \in \N$, and all $x \in [1,N_n-h]$,
\begin{align}\label{eqn:boundongdoubleprime} C_1 \ h \ n^{-(3-\sigma_1)} \leq \left| g_{n,h}''(x) \right| \leq C_2 \ h \ n^{-(3-\sigma_2)}.\end{align}
By the MVT, $g_{n,h}''(x) = h g_n'''(\xi_x)$ for some $\xi_x \in (x, x+h)$, so it suffices to show that for all $h \in \N$, all sufficiently large $n \in \N$, and all $x \in [1,N_n]$,
\begin{align}\label{eqn:boundongtripleprime} C_1 \ n^{-(3-\sigma_1)} \leq \left| g_{n}'''(x) \right| \leq C_2 \ n^{-(3-\sigma_2)}.\end{align}
Writing $g_n'''(x)$ explicitly reveals that
\[g_n'''(x) = g_n(x) (\log n)^3 \left( \frac{f_n'(x)}{f_n(x)} \right)^{3} \left(\frac{\varchange a {f_n(x)}^6}{-(\log a)^3} - \frac{r_n(x)}{\log n} \right),\]
where $r(x)$ is a sum of nine terms, five of which are of the form $c (\log a)^{i} (\log n)^{-(3-i)}\allowbreak (\log f_n(x))^{-j}$ where $c \in \{2,3,6\}$, $i \in \{1,2\}$ and $j \in \{2,3,4,5\}$, and four of which are of the form $c (\log a)^{i} (\log n)^{-(3-i)} f_n(x)^{j}  f_n^{(j+1)}(x) f_n'(x)^{-(j+1)}\allowbreak (\log f_n(x))^{-k}$ where $c \in \{1,-3,-6\}$, $i, j \in \{1,2\}$, and $k \in \{2,3,4\}$. By the bounds on the derivates of $f_n$, for sufficiently large $n$, $|r_n(x)|$ is bounded from above uniformly in $x \in [1,N_n]$. The inequality in (\ref{eqn:boundongtripleprime}) follows for $n$ sufficiently large since for all $x \in [1,N_n]$, the terms $f_n(x)$ and $\varchange a {f_n(x)}^6 / -(\log a)^3$ are bounded from above and away from 0, $n^{\sigma_1} \leq \big|g_n(x) (\log n)^3 \big| \leq n^{\sigma_2}$ and $(2n)^{-1} \leq |f_n'(x)| \leq 2 n^{-1}$.

To prove (\ref{eqn:equidistributiontoshow}), it suffices by Weyl's Criterion (\cite{kuipersbook}, Chapter 1, Theorem 2.1) to show that for all $b \in \Z \setminus \{0\}$,
\[\frac 1{N_n} \sum_{i=1}^{N_n} e \big(b g_n(i) \big) \longrightarrow 0 \quad \text{as} \quad n \to \infty,\]
where $e(x) = e^{2 \pi i x}$. By the van der Corput Difference Theorem (\cite{kuipersbook}, Chapter 1, Theorem 3.1) and another application of Weyl's Criterion, it suffices to prove that for all $h \in \N$ and for all $b \in \Z \setminus \{0\}$,
\begin{align}\label{eqn:vdcorputdifference}\frac 1{N_n-h} \sum_{i=1}^{N_n-h} e \big(b g_{n,h}(i) \big) \longrightarrow 0 \quad \text{as} \quad n \to \infty.\end{align}
An exponential sum estimate (\cite{kuipersbook}, Chapter 1, Theorem 2.7) gives us that
\[\frac 1{N_n-h} \left| \sum_{i=1}^{N_n-h} e \big(b g_{n,h}(i) \big) \right| \leq \left( \frac{|b| \ |g_{n,h}'(N_n-h) - g_{n,h}'(1) | +2}{N_n-h} \right) \left( \frac{4}{\sqrt{|b| \de}} +3 \right),\]
where $\de = C_1 h n^{-(3-\sigma_1)}$ from (\ref{eqn:boundongdoubleprime}). By the MVT and the upper bound from (\ref{eqn:boundongdoubleprime}), we see the right hand side is bounded from above for sufficiently large $n$ by
\[\left( |b|C_2 h \ n^{-(3-\sigma_2)} + \frac 2{N_n-h} \right) \left( \frac{4n^{(3-\sigma_1)/2}}{\sqrt{|b| C_1 h }} +3 \right) \ll \frac{n^{(3-\sigma_1)/2}}{n \sqrt{|b| h}},\]
where the implicit constant depends on $a$, $\gamma$, $\eta_1$, and $\eta_2$. The limit in (\ref{eqn:vdcorputdifference}) follows since $(3-\sigma_1)/2 < 1$.
\end{proof}

Now we can deduce Theorem \ref{thm:firstDAform}$'$ from Theorem \ref{thm:mainthm}.

\begin{proof}[Proof of Theorem \ref{thm:firstDAform}$'$]
Let $\Theta \subseteq (a,1)$ be the set of those $\theta$ satisfying the conclusion of Theorem \ref{thm:firstDAform}$'$. We will show that $\Theta$ is of full Lebesgue measure by showing separately that it has full measure in the intervals $(a,\sqrt a)$ and $(\sqrt a,1)$.

To show that $\Theta \cap (a,\sqrt a)$ is of full measure, we will cover $(a,\sqrt a)$ by short intervals and apply Theorem \ref{thm:mainthm} to each one. We will define $(\psi_n)_{n \in \N}$, $(\varphi_n)_{n \in \N}$, and $(\rho_n)_{n \in \N}$ so that every $\theta \in (a,\sqrt{a})$ for which (\ref{eqn:systemnew}) is solvable is a $\theta$ for which (\ref{eqn:systemtwo}) is solvable.

Note that $(a,\sqrt{a})$ can be covered by open intervals of the form $(j_1,j_2)$ with closure contained in $(a,\sqrt{a})$ with the property that
\begin{align}\label{eqn:intproperty}\varchange a{j_2} - \varchange a{j_1} < 2 - \varchange a{j_2}.\end{align}
Let $\openint = (j_1,j_2)$ be such an interval. We will define $(\psi_n)_{n \in \N}$, $(\varphi_n)_{n \in \N}$, and $(\rho_n)_{n \in \N}$ and verify that the conditions of Theorem \ref{thm:mainthm} hold.

For $n \in \N$, put
\[\psi_n = \frac{\min(c,10^{-1})}{n^{\varchange a {j_2}-1}}.\]
Because $0 < \varchange a {j_2}-1 < 1$, it follows from the discussion just after the statement of Theorem \ref{thm:mainthm} that conditions (A1) through (A4) and condition (B3) (with $\psi_p$ as the maximum) hold for this choice of $\psi_n$.

Define $\cent_n: \overline{J} \to \R$ to be
\[\cent_n(\theta) = \frac{\kappa \theta}{\varchange a \theta n^{\varchange a \theta - 1}}.\]
Since $t_a$ is infinitely differentiable on $\overline{J}$ with derivatives bounded uniformly from above and away from $0$, the function $\cent_n$ is $C^3$ on $\overline J$ and for $i \in \{0,1,2,3\}$,
\begin{align}\label{eqn:supnormestimates}\supnormopenint{\cent_n^{(i)}} \ll \frac{(\log n)^{i}}{n^{\varchange a \theta - 1}}.\end{align}
Condition (B1) in Theorem \ref{thm:mainthm} follows from (\ref{eqn:supnormestimates}) with $i=0$. To verify (B2), note that
\[\cent_n'(\theta) = \frac{\kappa}{n^{\varchange a \theta - 1}} \left( \frac{\log n}{\log \theta} + \frac{1+\log \theta}{\log a} \right).\]
It follows that for $n$ sufficiently large, $\supnormopenint{\cent_n'} = |\cent_n'(j_1)|$. This sequence is eventually decreasing, verifying (B2).  Condition (B3) with $\supnormopenint{\cent_p'}$ as the max follows from (\ref{eqn:supnormestimates}), a calculation, and the inequality in (\ref{eqn:intproperty}). Condition (B4) follows from a calculation and the inequality $t_a(j_2) - t_a(j_1) < 1$, which follows from the fact that $1 < t_a(j_1) < t_a(j_2) < 2$.

Define $\rho_n: \overline{J} \to \R$ to be $\rho_n(\theta) = \gamma n^{\varchange a \theta}$. Let $J' = (j_1',j_2')$ be a proper subinterval of $J$, and note that $\|\rho_n\|_{\overline{J'},\infty} \ll n^{\varchange a {j_2'}} \log n$. In order to verify (C1), we have only to note that the conditions of Lemma \ref{lem:equid} are met by the inequalities in (\ref{eqn:supnormestimates}). Condition (C2) follows from a calculation and the fact that $t_a(j_2') < t_a(j_2)$.

It is simple to check that every $\theta \in (a,\sqrt{a})$ for which (\ref{eqn:systemnew}) is solvable is a $\theta$ for which (\ref{eqn:systemtwo}) is solvable.  Since the conditions of Theorem \ref{thm:mainthm} hold, the set $\Theta \cap \openint$ is of full measure. Since $(a,\sqrt{a})$ was covered by such intervals $\openint$, the set $\Theta \cap (a,\sqrt{a})$ is of full measure.

To show that the set $\Theta \cap (\sqrt a,1)$ is of full measure, we will show that for all $\intthree > \sqrt a$, the set $(\intthree,1) \setminus \Theta$ has zero measure. Define $\varphi_n$ as above, and let $\sigma \in (1,\varchange a\intthree -1)$. Put

\[H_n = \left\{ \theta \in (\intthree,1) \ \middle | \ \big \| \theta n + \cent_n(\theta) \big\| \leq \frac{1}{n^{\sigma} } \right\}.\]
If $\theta \in (\intthree,1) \setminus \Theta$, then (\ref{eqn:systemtwo}) is solvable, meaning that for infinitely many $n\in \N$, $\fracdist {\theta n + \cent_n(\theta)} \leq 1 \big/ n^{\sigma}$. It follows that
\begin{align}\label{eqn:boundfromabovebylimsup}(\intthree,1) \setminus \Theta \subseteq \limsup_{n \to \infty} H_n.\end{align}
Just as in (\ref{eqn:Enisaunionofintervals}), $H_n$ is covered by a union of $\ll (1-\intthree)n$ intervals, each of length $\ll n^{-(\sigma+1)}$. Since $\sigma > 1$, $\sum_{n=1}^\infty \lambda(H_n) < \infty$. By the first Borel-Cantelli Lemma, $\limsup_{n \to\infty}H_n$ has zero measure, so (\ref{eqn:boundfromabovebylimsup}) implies that $(\intthree,1) \setminus \Theta$ has zero measure.
\end{proof}

Finally, we deduce Corollary \ref{cor:maincortwo} from Theorem \ref{thm:mainapplication}.

\begin{proof}[Proof of Corollary \ref{cor:maincortwo}]
Denote by $Q(\al)$ the limit quotient set in (\ref{eqn:limitquotentset}). Note that $a \in Q(\al)$ if and only if the linear equation $y=ax$ is solvable in $\psa$. By Theorem \ref{thm:mainapplication}, the set
\[ \bigcap_{a \in \Qplus \setminus \{1\}} \big\{ \al \in (1,2) \ \big| \ \text{the equation $y=ax$ is solvable in $\psa$} \big\}\]
is of full measure in the interval $(1,2)$, proving that for Lebesgue-a.e. $\al < 2$, $Q(\al) = \Q_+$. On the other hand, Theorem \ref{thm:mainapplication} gives that the set
\[ \bigcup_{a \in \Qplus \setminus \{1\}} \big\{ \al \in (2,\infty) \ \big| \ \text{the equation $y=ax$ is solvable in $\psa$} \big\}\]
is of zero measure in $(2,\infty)$, proving that for Lebesgue-a.e. $\al > 2$, $Q(\al) = \{1\}$.
\end{proof}

\bibliographystyle{abbrv}
\bibliography{thesisbib}

\end{document}